\documentclass[12pt,a4paper]{amsart}
\usepackage{url}
\usepackage{amsmath,amsthm,amsfonts,amssymb,latexsym,undertilde}
\usepackage{graphicx}
\usepackage{rotating}

\date{\today}

\newtheorem{theorem}{Theorem}[section]

\newtheorem{lemma}[theorem]{Lemma}

\theoremstyle{definition}

\newcommand{\U}{\mathcal U}

\newcommand{\w}{\omega}

\newcommand{\G}{\mathcal{G}}

\newcommand{\M}{\mathcal{M}}
\newcommand{\F}{\mathcal{F}}
\newcommand{\T}{\mathcal{T}}

\newcommand{\V}{\mathcal{V}}

\newcommand{\add}{\operatorname{add}}

\newcommand{\cov}{\operatorname{cov}}

\newcommand{\ttt}{\mathfrak t}
\newcommand{\bb}{\mathfrak b}
\newcommand{\dd}{\mathfrak d}

\newcommand{\e}{\varepsilon}

\newcommand{\la}{\langle}
\newcommand{\ra}{\rangle}

\newcommand{\nothing}[1]{}

\title[Productiv Lindel\"ofness and  Hurewicz' property]{Productively
Lindel\"of spaces and the covering property of Hurewicz}

\author{Du\v{s}an Repov\v{s} and  Lyubomyr Zdomskyy}

\address{Faculty of Education, and Faculty of Mathematics and Physics,
University of Ljubljana, P. O. Box 2964, Ljubljana, Slovenia 1001.}
\email{dusan.repovs@guest.arnes.si}
\urladdr{http://www.fmf.uni-lj.si/\~{}repovs/index.htm}

\address{Kurt G\"odel Research Center for Mathematical Logic,
University of Vienna, W\"ahringer Stra\ss e 25, A-1090 Wien,
Austria.} \email{lzdomsky@gmail.com}
\urladdr{http://www.logic.univie.ac.at/\~{}lzdomsky/}

\thanks{The first author acknowledges support by SRA grants P1-0292-0101 and J1-4144-0101.
The second author would  like to acknowledge the generous support
of FWF grant M1244-N13.}

\begin{document}

\begin{abstract}
We prove that
under certain set-theoretic assumptions
 every productively Lindel\"of space has the Hurewicz covering property,
 thus improving
upon some earlier results  of  Aurichi and Tall.
\end{abstract}

\maketitle

\section{Introduction}
A topological space $X$ is called \emph{productively Lindel\"of } if
$X\times Y$ is Lindel\"of for every Lindel\"of space $Y$. This
terminology was introduced in \cite{BarKenRap07}, but the concept
itself goes back at least to the classical work  of
Michael \cite{Mic63} who proved that under CH the space of irrational numbers is
not productively Lindel\"of. The natural question of whether an
additional set-theoretic hypothesis is needed here has become known
as \emph{Michael's problem} and is still open. Thus we are at the moment far
from a satisfactory  understanding of productive Lindel\"ofness,
even for subspaces of the Baire space $\w^\w$.

In a stream of recent papers of  Tall and
collaborators  it was proven  that under certain
equalities between cardinal characteristics all metrizable
productively Lindel\"of spaces have strong covering properties
close to the $\sigma$-compactness.   In  modern terminology
such covering properties are called \emph{selection principles} and constitute
 a rapidly growing area of general topology (see e.g.,  \cite{Tsa07}).

Trying to describe  $\sigma$-compactness in terms of open covers, Hurewicz
 \cite{Hur25} introduced
  the following property, nowadays called \emph{ the Menger property}, which was
historically the first selection principle: a topological space
$X$ is said to have this  property if for every sequence $\la \U_n : n\in\omega\ra$
of open covers of $X$ there exists a sequence $\la \V_n : n\in\omega \ra$ such that
each $\V_n$ is a finite subfamily of $\U_n$ and the collection $\{\cup \V_n:n\in\omega\}$
is a cover of $X$. The current name (the Menger property) is used because Hurewicz
proved in   \cite{Hur25} that for metrizable spaces his property is equivalent to
one basis property considered by Menger in \cite{Men24}, see \cite{BabSch03_04} for more information
on combinatorial properties of bases.
If in the definition above we additionally require that $\{\cup\V_n:n\in\w\}$
is a \emph{$\gamma$-cover} of $X$
(this means that the set $\{n\in\w:x\not\in\cup\V_n\}$ is finite for each $x\in X$),
then we obtain the definition of the Hurewicz covering property introduced
in \cite{Hur27}. Contrary to a conjecture of Hurewicz
the class of  metrizable spaces having the Hurewicz property
 appeared  to be much wider than the class of $\sigma$-compact spaces \cite[Theorem~5.1]{COC2}
(see also \cite{BarTsa06,Tsa07}).

By  \cite[Theorem~23]{AurTal12}
and \cite[Theorem~18]{Tal11_tappl}, every  productively Lindel\"of space has the
Hurewicz property if $\dd=\w_1$ or $\add(\M)=\mathfrak c$.
The following theorem implies both of these results.

\begin{theorem}\label{main2}
If $\add(\M)=\dd$, then every productively Lindel\"of space has the Hurewicz
property.
\end{theorem}

If $\bb=\w_1$, then by \cite[Corollary~4.5]{AlaAurJunTal11}
every  productively Lindel\"of space has the Menger
property. The following result shows that
with a help of an  additional combinatorial assumption about
filters on $\w$ we can actually infer the Hurewicz property.

\begin{theorem} \label{main1}
If $\bb=\w_1$ and the Filter Dichotomy holds, then every productively Lindel\"of
space has the Hurewicz property.
\end{theorem}

The \emph{Filter Dichotomy} is the statement that for any non-meager
filters $\F,\G$ on $\w$ there exists a monotone surjection
$\phi:\w\to\w$ such that $\phi(\F)=\phi(\G)$.
Here we consider  filters on $\w$
  with the topology inherited from $\mathcal P(\w)$,
the latter being  identified with the Cantor space $2^\w$ via  characteristic functions.

By \cite[Theorems~1,2]{BlaLaf89}, the Filter Dichotomy (abbreviated as FD in the sequel)
holds in the Miller model, and hence the premises  of Theorem~\ref{main1}
do not imply those of Theorem~\ref{main2},  for the
values of cardinal characteristics in some standard iteration models see \cite[p.~480]{Bla10}.

It has been noted in \cite{Tal11_qagt} that
 the three progressively weaker hypotheses: \emph{CH, $\mathfrak d=\w_1$,} and
\emph{$\w^\w$ is not productively Lindel\"of}, imply the
respectively weaker conclusions about metrizable productively
Lindel\"of spaces: \emph{$\sigma$-compact, the Hurewicz property,}
and \emph{the Menger property}. In \cite[Problem 3.13]{Tal11_qagt}
it is asked whether  the stronger hypotheses are necessary in order
to obtain the stronger conclusions. Theorems~\ref{main2} and
\ref{main1} may be thought of as a tiny step towards the solution of
this problem.

In these results we do not assume that $X$ satisfies any separation axioms.
This generality is achieved with the help of set-valued maps,
which by the methods developed in \cite{Zdo05} lead to a reduction
to subspaces of the Baire space.
For the definitions of cardinal characteristics  used in
this paper we refer the reader to \cite{Bla10} or \cite{Vau90}.

\section{Proofs}

 Theorem~\ref{main2} will be proved by adding ``an $\e$'' to the
following deep result:

\begin{theorem}\label{moo}
 Let $X$ be a topological space which admits a compactification whose remainder
is Lindel\"of. If there exists a cardinal $\kappa$ of uncountable cofinality
and an  increasing sequence $\la X_\alpha:\alpha<\kappa\ra$ of principal subsets of
$X$ such that $\bigcup_{\alpha<\kappa}X_\alpha=X$,  for every compact $K\subset X$
there exists an ordinal $\alpha$ such that $K\subset X_\alpha$,
and the minimal ordinal with this property has  countable cofinality,
 then $X$ is not productively Lindel\"of.
\end{theorem}

Theorem~\ref{moo} can be proved by almost verbatim
repetition of a part of the proof
of \cite[Theorem~1.2]{Moo99}.
\medskip

\noindent\textit{Proof of Theorem~\ref{main2}.} \
Given  a productively Lindel\"of space $X$, we shall show that it has
the Hurewicz property.

 Combining \cite[Lemma~1 and Theorem~2]{Zdo05}
we conclude that $X$ has the  Hurewicz property if and only if all images of $X$
under compact-valued  upper semicontinuous maps $\Phi:X\Rightarrow\w^\w$ have it.
Since any such image of $X$ is productively Lindel\"of, it is enough to show
that productively Lindel\"of subspaces of $\w^\w$ have the Hurewicz property.
Therefore we shall  assume that $X\subset\w^\w$.
 Suppose to the contrary that  $X$ does not have the Hurewicz property.
Using  \cite[Theorem~4.3]{COC2} and passing to a homeomorphic copy of $X$, if necessary,
we may additionally assume that $X$ is unbounded with respect to $\leq^*$. The proof
will be completed as soon as we derive a contradiction
with  $\add(\M)=\dd$.

Let $D=\{d_\alpha:\alpha<\dd\}$ be a dominating family.
Since $\add(\M)\leq\bb\leq\dd$ \cite{Vau90},
we conclude that $\bb=\dd$. For every $\alpha<\bb $ set $X_\alpha=\{x\in X:x\leq^*d_\xi$
for some $\xi<\alpha\}$. Since $X$ is unbounded and $D$ is dominating,
 $X_\alpha\neq X$ for all $\alpha$ and $\bigcup_{\alpha<\dd}X_\alpha=X$.
Given a compact $K\subset X$ we can find $\alpha<\dd$ such that all elements of
$K$ are bounded by $f_\alpha$, and hence  $K\subset X_\alpha$.
Observe that each $X_\alpha$ is a union of a family $\V_\alpha$
of fewer than $\cov(\M)$ many
closed subsets of $X$, where
$$ \V_\alpha=\big\{\{x\in X:x(k)\leq d_\xi(k)\mbox{ for all }k\geq n\}: \xi<\alpha,n\in\w\big\}. $$
By the same argument as in the proof of \cite[Lemma~2]{Als90} we can show that
there exists a countable subfamily $\V'$ of $\V_\alpha$ covering $K$.
It follows from the above that the minimal ordinal $\beta$ such that $K\subset X_\beta$
has countably cofinality. Therefore the sequence $\la X_\alpha:\alpha<\dd\ra$ satisfies
the premises of Theorem~\ref{moo}, and hence $X$ is not productively Lindel\"of.
\hfill $\Box$
\medskip

 Theorem~\ref{main1} is a direct consequence of
Theorem~\ref{main1_tech} below, where the FD is weakened to the following
assumption:
\begin{itemize}
 \item[$(*)$] For every non-meager filter $\G$ there exists an unbounded tower
$\T$ of cardinality $\bb$ and  a monotone surjection $\phi:\w\to\w$ such that
$\phi(\T)\subset \phi(\G)$.
\end{itemize}

We recall that a \emph{tower of cardinality $\kappa$}  is a set
$\T \subset [\w]^\w$  which can be enumerated
 as $\{T_\alpha:\alpha<\kappa \}$, such that for all
$\alpha<\beta<\kappa$, $T_\beta\subset^* T_\alpha$ and $T_\alpha\not\subset^* T_\beta$,
where $A\subset^* B$ means $|A\setminus B|<\w$.
An \emph{unbounded tower of cardinality $\kappa$} is an unbounded with respect to $\leq^*$ set $\T\subset [\w]^\w$ which is a
tower of cardinality $\kappa$ (here we identify each element of $[\w]^\w$ with its increasing
enumeration). It is an easy exercise to show that
$\ttt=\bb$ if  and only if  there is an unbounded tower of
cardinality $\ttt$.

We shall use the following fundamental result of Talagrand \cite{Tal82}.

\begin{theorem} \label{tal}
A filter $\F$ is meager if  and only if there exists an increasing sequence $\la n_k:k\in\w\ra$
of natural numbers such that each  $F\in\F$ meets  all but finitely many
intervals $[n_k,n_{k+1})$.
\end{theorem}

The following lemma implies that Theorem~\ref{main1_tech} is indeed an improvement
of Theorem~\ref{main1}.

\begin{lemma} \label{l2}
If $\bb=\ttt$ and   the FD  holds, then $(*)$ holds as well.
\end{lemma}
\begin{proof}
Let $\T$ be an unbounded tower of cardinality $\bb$ and $\F$ be a non-meager filter.
The FD yields a monotone surjection $\phi:\w\to\w$ such that
$\phi(\la \T\ra)=\phi(\F)$. Then $\phi(\T)\subset\phi(\F)$ and $\phi(\T)$ is an unbounded tower of cardinality $\bb$.
\end{proof}

A set $S=\{f_\alpha : \alpha<\bb\}$ is called  a \emph{$\bb$-scale} if $S\subset\w^\w$,
all elements of $S$ are increasing,
$S$ is unbounded with respect to $\leq^*$, and  $f_\alpha\leq^*  f_\beta$ for each $\alpha<\beta<\bb$.
It is easy to see that a $\bb$-scale always exists.

 Let $\chi:[\w]^\w\to \w^\w$ be the map assigning to each infinite subset $A$ of $\w$ its enumeration $e_A\in\w^\w$
(i.e., $e_A(n)$ is the $n$th element of $A$). Then $\chi $ is an embedding of $[\w]^\w$ into $\w^\w$
which maps every unbounded tower of cardinality $\bb$ onto a $\bb$-scale.
It is a  direct consequence of
\cite[Corollary~2.5]{AlaAurJunTal11} that if $\bb=\w_1$ and $X$ is a productively Lindel\"of subspace of $\w^\w$, then
$B\not\subset X$ for any $\bb$-scale $B$. As a corollary we get the following

\begin{lemma} \label{l3}
If $\bb=\w_1$ and $X$ is a productively Lindel\"of subspace of $[\w]^\w$, then
$\T\not\subset X$ for any unbounded tower  $\T$.
\end{lemma}

In the proof of Theorem~\ref{main1_tech} we shall use set-valued maps,
see \cite{RepSem98}.
By a \emph{set-valued map} $\Phi$ from a set $X$ into a set $Y$ we
understand a map from $X$ into the power-set $\mathcal P(Y)$ of $Y$
 and write $\Phi :
X\Rightarrow Y$. For a subset $A$ of $X$ we define $\Phi(A) =
\bigcup_{x\in A}\Phi(x) \subset Y$. A set-valued map $\Phi$ from a
topological spaces $X$ to a topological space $Y$ is said to be
\begin{itemize}
\item  \emph{compact-valued}, if $\Phi(x)$  is compact for every $x \in X$;
\item \emph{upper semicontinuous}, if for every open subset $V$ of $Y$ the
set
$\Phi^{-1}_\subset (V) = \{x \in  X : \Phi(x) \subset V\}$ is open in $X$.
\end{itemize}

A family $\F\subset [\w]^\w$ is called \emph{centered}, if $\cap\F_1\in [\w]^\w$ for every $\F_1\in [\F]^{<\w}$.
For a centered family $\F$ we shall denote by $\la\F\ra$ the smallest non-principal filter on $\w$
containing $\F$. In other words, $\la\F\ra=\{A\subset\w:\cap\F_1\subset^* A$ for some $\F_1\in [\F]^{<\w}\}$.

The following easy lemma is a direct consequence of \cite[Claim 3.2(2)]{RepZdo12}.

\begin{lemma} \label{l1}
If a centered family $\F$ is an image of a topological space $X$ under a  compact-valued
upper semicontinuous map, then $\la\F\ra$ is a countable union of
images of finite powers of  $X$ under   compact-valued
upper semicontinuous maps.
\end{lemma}

Let $\U$ be a family of subsets of a set $X$. A subset $A$ of $X$ is called
\emph{$\U$-bounded} if $A\subset\cup\V$ for some finite $\V\subset\U$. $\U$ is
called an \emph{$\w$-cover} of $X$ if $X\not\in\U$ and for every finite $F\subset X$ there
exists $U\in\U$ such that $F\subset U$.

\begin{theorem}\label{main1_tech}
If $\bb=\w_1$ and the assumption $(*)$ from above holds, then every productively Lindel\"of
space has the Hurewicz property.
\end{theorem}
\begin{proof}
Let $X$ be a productively Lindel\"of space and let $\la \U_n:n\in\w
\ra$ be a sequence of open covers of $X$. Let us write $\U_n$ in the
form $\{U^n_k:k\in\w\}$. Observe that there is no loss of
generality in assuming  $U^{n+1}_k\subset U^n_k\subset U^n_{k+1}$   for all
$n,k\in\w$.

The equality $\bb=\w_1$ implies that $\w^\w$ is not productively Lindel\"of \cite[Remark 10.5]{vDo84},
and hence by \cite[Proposition~3.1]{RepZdo12}
 all productively Lindel\"of spaces have the Menger property.
 Since the class of productively Lindel\"of spaces is closed under
finite products, we conclude that all finite powers of $X$ have the Menger property.
 By \cite[Theorem~3.9]{COC2} there exists a sequence
$\la\V_n:n\in\w\ra$ such that $\V_n\in [\U_n]^{<\w}$ for all $n$ and
$\V=\bigcup_{n\in\w}\V_n$ is an $\w$-cover of $X$. It follows from our assumptions
on $\U_n$'s that
 we may assume  $|\V_n|=1$  for
all $n\in\w$, i.e.,
 $\V_n=\{U^n_{k_n}\}$ for some
$k_n\in\w$.

 For every $x\in X$ we shall denote by $I_\V(x)$ the set
  $\{n\in\w:x\in U^n_{k_n}\}$. By \cite[Lemma~2]{Zdo05}
 the  set $$\F=\{A\subset\w:I_\V(x)\subset^* A \mbox{ for some }x\in X\}$$
 is a countable union of images of $X$ under compact-valued upper
 semicontinuous maps. Therefore by Lemma~\ref{l1}
 $\G=\la\F\ra$
 is a countable union of images of finite powers of $X$ under compact-valued upper
 semicontinuous maps. Since $\V$ is an $\w$-cover of $X$, we conclude
that $\G\subset [\w]^\w$. By the methods of \cite[\S~3]{RepZdo12}
 it also follows that $\G$ is productively Lindel\"of.
 Two cases are possible.

1. $\G$ is meager. Then by Theorem~\ref{tal}  there exists an increasing  sequence $\la m_n:n\in\w  \ra$ of integers
such that every element of $\U$ meets all but finitely many intervals $[m_n,m_{n+1})$.
It suffices to observe that
$U_n=\bigcup\{U^m_{k_m}:m\in [m_n,m_{n+1})\}$ is $\U_n$-bounded for every $n\in\w$
and the sequence $\la U_n:n\in\w \ra$ is a $\gamma$-cover of $X$.

2. $\G$ is non-meager. By the assumption $(*)$ from above we can find a monotone surjection
$\phi:\w\to\w$ and an unbounded tower $\T\subset [\w]^{\w}$ of cardinality $\bb$
such that $\phi(\U)\supset \T$,
 which contradicts Lemma~\ref{l3} and thus completes our proof.
\end{proof}

\end{document}